\documentclass[a4paper,11pt]{article}

\usepackage{blindtext,titlefoot}

\usepackage{amsfonts,amsthm,amssymb,amsmath}
\usepackage[title,titletoc,toc]{appendix}
\usepackage[utf8]{inputenc}
\usepackage[numbers]{natbib}

 \usepackage{amsthm,amsmath}
 \usepackage[numbers]{natbib}

\usepackage{lineno,hyperref}
\usepackage[affil-it]{authblk}
\usepackage{xcolor}
\usepackage{tikz}
\newcommand\COMP{\hbox{C\kern -.58em {\raise .54ex \hbox{$\scriptscriptstyle |$}}
\kern-.55em {\raise .53ex \hbox{$\scriptscriptstyle |$}} }}
\newcommand\NN{\hbox{I\kern-.2em\hbox{N}}}
\newcommand\RR{\hbox{I\kern-.2em\hbox{R}}}
\newcommand\sRR{{\it \hbox{I\kern-.2em\hbox{R}}}}
\newcommand\QQ{\hbox{I\kern-.53em\hbox{Q}}}
\newcommand\PP{\hbox{I\kern-.53em\hbox{P}}}
\newcommand\EE{\hbox{I\kern-.53em\hbox{E}}}
\newcommand\ZZ{{{\rm Z}\kern-.28em{\rm Z}}}
\newcommand\be{\begin{equation}}
\newcommand\ee{\end{equation}}

\DeclareMathOperator*{\essinf}{ess\,inf}

\DeclareMathOperator*{\esssup}{ess\,sup}

\newtheorem{theorem}{Theorem}[section]

\newtheorem{proposition}[theorem]{Proposition}
\newtheorem{remark}[theorem]{Remark}

\newtheorem{lemma}[theorem]{Lemma}

\newtheorem{definition}[theorem]{Definition}

\makeatletter
\newcommand*\bigcdot{\mathpalette\bigcdot@{.5}}
\newcommand*\bigcdot@[2]{\mathbin{\vcenter{\hbox{\scalebox{#2}{$\m@th#1\bullet$}}}}}
\makeatother
\newcommand{\is}{\bigcdot }

\def \Lbrack {[\![}
\def \Rbrack {]\!]}

\numberwithin{equation}{section}

\begin{document}
\title{Optimal stopping problem under random horizon}

\author[1]{Tahir Choulli\thanks{tchoulli@ualberta.ca}}
\author[2]{Safa' Alsheyab\thanks{alsheyab@ualberta.ca}}

\affil[1]{Mathematical and Statistical Sciences Dept., University of  Alberta, Edmonton, AB, Canada}

\affil[2]{Department of Mathematics and Statistics, Jordan University of Science and Technology, Irbid 22110, Jordan}

\renewcommand\Authands{ and }





\maketitle

{\bf P.S. This version is the first half of the old version ``Optimal stooping problem: Mathematical Structures and Linear RBSDEs,  ArXiv:2301.09836" and which has been split into two parts. The second part focuses on linear RBSDEs under random horizon.}

\begin{abstract}
This paper considers a pair $(\mathbb{F},\tau)$, where $\mathbb{F}$  is a filtration representing the ``public" flow of information which is available to all agents overtime, and $\tau$ is a random time which might not be an $\mathbb{F}$-stopping time. This setting covers the case of credit risk framework where $\tau$ models the default time of a firm or client,  and the setting of life insurance where $\tau$ is the death time of an agent. It is clear that random times can not be observed before their occurrence. Thus the larger filtration $\mathbb{G}$, which incorporates $\mathbb{F}$ and makes $\tau$ observable, results from the progressive enlargement of $\mathbb{F}$ with $\tau$. For this informational setting, governed by $\mathbb{G}$, we analyze the optimal stopping problem in three main directions. The first direction consists of characterizing the existence of the solution to this problem in terms of $\mathbb{F}$-observable processes. The second direction lies in deriving the {\it mathematical structures} of the value process of this control problem,  while the third direction singles out the associated optimal stopping problem under $\mathbb{F}$. These three aspects allow us  to quantify deeply how $\tau$ impact the optimal stopping problem, while they are also vital for  studying reflected backward stochastic differential equations which arise {\it naturally} from pricing and hedging of vulnerable claims.
\end{abstract}


\section{Introduction}
In this paper, we consider a complete probability space $\left(\Omega, {\cal F}, P\right)$, on which we consider given a complete and right-continuous filtration $\mathbb F:=({\cal F}_t)_{t\geq0}$. Besides this initial system $(\Omega, {\cal F}, \mathbb F, P)$, we consider an arbitrary random time $\tau$ (i.e. an ${\cal{F}}$-measurable random variable with values in $[0,+\infty]$) that might not be an $\mathbb F$-stopping time. As in most applications, such as life insurance and credit risk where $\tau$ is the death time and the default time respectively, $\tau$ is observable when  it occurs only and can not be seen before. Thus, the flow of information that incorporates both $\tau$ and  $\mathbb{F}$, which will be denoted throughout the paper by $\mathbb{G}=({\cal{G}}_t)_{t\geq 0}$, makes $\tau$ a stopping time and is known in the literature as the progressive enlargement of $\mathbb{F}$ with $\tau$. For this new system  $(\Omega, {\cal F}, \mathbb{G}, P)$, { \it{our objective}} consists of analyzing the following problem
\begin{equation}\label{OSP4G}
{\cal S}_{\sigma}^{\mathbb G}:=\esssup_{\theta\in{\cal{J}}_{\sigma}}E\left[X^{\mathbb{G}}_{\theta}\ \big|\quad{\cal{G}}_{\sigma}\right],
\end{equation}
 where ${\sigma}$ is a $\mathbb{G}$-stopping time, and ${\cal{J}}_{\sigma}$ is the set of $\mathbb{G}$-stopping times that are finite and greater or equal to $\sigma$. $X^{\mathbb G}$  is a $\mathbb{G}$-optional process representing the reward, satisfying ``some integrability condition", and is stopped at $\tau$ (i.e. it does not vary after $\tau$). {{The essential supremum of an arbitrary family of random variables is the smallest random variable which is an upper bound for each element of this family almost surely, see \cite{Neveu} for more details and related properties.}} \\
 This problem is known as the optimal stopping problem, and it is an example of stochastic control problem. For more details about its origin, its applications and its evolution, we refer the reader to \cite{Arai,ElKaroui, Quenez1,Peskir,Shiryaev} and the references therein to cite a few. Herein, we address this optimal stopping problem, and we aim to measure the impact of $\tau$ on this problem in many aspects. In particular, for this problem we answer the following questions.
 \begin{enumerate}
 \item Can we associate to (\ref{OSP4G}) an optimal stopping problem under $\mathbb{F}$ with reward $\widetilde{X}^{\mathbb{F}}$ and value process ${\cal S}^{\mathbb{F}}$? 
 \item How the two pairs $(X^{\mathbb{G}},{\cal S}^{\mathbb G})$ and $(\widetilde{X}^{\mathbb{F}},{\cal S}^{\mathbb{F}})$ are connected to each other? 
 \item What are the structures in ${\cal S}^{\mathbb G}$ induced by $\tau$? 
 \item How the maximal (minimal) optimal times of (\ref{OSP4G}) and their $\mathbb{F}$-optimal stopping problem counter parts are related to each others?\end{enumerate}
 {\it{One of the direct applications}} of our optimal stopping problem under random horizon $\tau$, lies in studying linear RBSDE having the form of  
\begin{align}\label{RBSDE1}
\begin{cases}
dY_{t}=-f(t)d(t\wedge\tau)-d(K_{t}+M_{t})+Z_{t}dW_{t\wedge\tau},\quad {Y}_{\tau}=\xi,\\
 Y\geq S\quad\mbox{on}\quad\Lbrack 0,\tau\Lbrack,\quad\mbox{and}\quad E\left[\displaystyle\int_{0}^{\tau}(Y_{t-}-S_{t-})dK_{t}\right]=0.
\end{cases}
\end{align}
Here  $\mathbb F$ is assumed to be generated by a Brownian motion $W$, $f$ is an $\mathbb{F}$-progressively measurable process (the driver rate), $\xi$ is a random variable, and $S$ is a  {{right-continuous with left-limits (RCLL for short hereafter)}} $\mathbb F$-adapted with values in $[-\infty, +\infty)$. {{The two processes $Y_{-}$ and $S_{-}$ are the left limits of $Y$ and $S$ respectively and are defined at the of Subsection \ref{Subsection1} for the sake of a smooth presentation)}}. For this direct application, we refer the reader to our earlier version of the whole work \cite{Choulli6}. The relationship between optimal stopping problem and RBSDEs is well understood nowadays, and we refer the reader to \cite{Quenez,Hamadane0,Hamadane1,Vanmaele1,Vanmaele2,QuenezSulem} and the references therein to cite a few. 

 This paper has four sections including the current one. The second section defines general notations, the mathematical setting of the random horizon $\tau$, and its preliminaries. The third section addresses the optimal stopping problem under stopping with $\tau$ in various aspects. {{We give our conclusion in the fourth section}}. The paper has an appendix where we prove our technical lemmas.
 
\section{Notation and the random horizon setting}

This section has two subsections. The first subsection defines general notation that will be used throughout the paper, while the second subsection presents the progressive enlargement setting associated to $\tau$ and its preliminaries. 
\subsection{General notation}\label{Subsection1}
By  ${\mathbb H}$ we denote an arbitrary  filtration that satisfies the usual conditions of completeness and right continuity.  For any process $X$, the $\mathbb H$-optional projection and  the $\mathbb H$-predictable projection, when they exist, will be denoted by $^{o,\mathbb H}X$  and $^{p,\mathbb H}X$ respectively. The set ${\cal M}(\mathbb H, Q)$ (respectively  ${\cal M}^{p}(\mathbb H, Q)$ for $p\in (1,+\infty)$) denotes the set of all $\mathbb H$-martingales (respectively $p$-integrable martingales) under $Q$, while ${\cal A}(\mathbb H, Q)$ denotes the set of all $\mathbb H$-optional processes that are  {{RCLL}} with integrable variation under $Q$. When $Q=P$, we simply omit the probability for the sake of simple notation.  For a {{$d$-dimensional $\mathbb H$-semimartingale $X$, by $L(X,\mathbb H)$ we denote the set of $\mathbb H$-predictable processes (either $d$-dimensional or one dimensional) that are $X$-integrable in the semimartingale sense.  For $\varphi\in L(X,\mathbb H)$, the resulting integral of $\varphi$ with respect to $X$ is denoted by $\varphi\is{X}$. If $\varphi$ is $d$-dimensional (respectively one dimensional), then $\varphi\is{X}$ is a one dimensional process (respectively $d$-dimensional process such as $I_{\Rbrack0,\sigma\Rbrack}\is{X}=X^{\sigma}-X_0$ for any stopping time $\sigma$). For more details about stochastic integral and its intrinsic calculus and notation, we refer the reader to \cite{DellacherieMeyer80}, \cite{YanBook}, and \cite{JS03}.}} For $\mathbb H$-local martingale $M$, we denote by $L^1_{loc}(M,\mathbb H)$ the set of $\mathbb H$-predictable processes $\varphi$ that are $M$-integrable and the resulting integral $\varphi\is M$ is an $\mathbb H$-local martingale. If ${\cal C}(\mathbb H)$ is a set of processes that are adapted to $\mathbb H$, then ${\cal C}_{loc}(\mathbb H)$ is the set of processes, $X$, for which there exists a sequence of $\mathbb H$-stopping times, $(T_n)_{n\geq 1}$, that increases to infinity and $X^{T_n}$ belongs to ${\cal C}(\mathbb H)$, for each $n\geq 1$.  The $\mathbb H$-dual optional projection and the $\mathbb H$-dual predictable projection of a process $V$ with finite variation, when they exist, will be denoted by  $V^{o,\mathbb H}$  and $V^{p,\mathbb H}$ respectively. For any real-valued $\mathbb H$-semimartingale $L$, we denote by ${\cal E}(L)$ the Dol\'eans-Dade (stochastic) exponential. It is the unique solution to $dX=X_{-}dL,\ X_0=1,$  given by
\begin{eqnarray}\label{DDequation}
 {\cal E}_t(L)=\exp\left(L_t-L_0-{1\over{2}}\langle L^c\rangle_t\right)\prod_{0<s\leq t}(1+\Delta L_s)e^{-\Delta L_s}.\end{eqnarray}
 Throughout the paper,  {{we consider the following notations: For any random time $\sigma$ and any process $X$, we denote by $X^{\sigma}$ the stopped process given by $X^{\sigma}_t:=X_{\sigma\wedge t},\ t\geq 0$. For any RCLL process $X$, we denote $X_{-}$ the process left limits of $X$ and is defined as $X_{-}=(X_{t-})_{t\geq 0}$, where $X_{0-}=X_0$ and $X_{t-}=\lim_{s\upuparrows t}X_s$ for $t>0$. A process $X$ is called a BMO $\mathbb{F}$-martingale if $X\in{\cal{M}}(\mathbb{F})$ (hence $X_t=E[X_{\infty}\big|{\cal{F}}_t]$) and there exists a constant $C>0$ such that $E[\vert X_{\infty}-X_{\sigma-}\vert\big|\mathcal{F}_{\sigma}]\leq C$ for any $\mathbb{F}$-stopping time $\sigma$. For more details about BMO martingales and their properties we refer the reader to \cite{DellacherieMeyer80}, or its English version \cite{DELLACHERIE}.}}
\subsection{The random horizon and the progressive enlargement of $\mathbb F$}
In addition to this initial model $\left(\Omega, {\cal F}, \mathbb F,P\right)$, we consider an arbitrary random time, $\tau$, that might not be an $\mathbb F$-stopping time. This random time is parametrized though $\mathbb F$ by the pair $(G, \widetilde{G})$, called survival probabilities or Az\'ema supermartingales, and are given by
\begin{eqnarray}\label{GGtilde}
G_t :=\ ^{o,\mathbb F}(I_{\Lbrack0,\tau\Lbrack})_t= P(\tau > t | {\cal F}_t) \ \mbox{ and } \ \widetilde{G}_t :=\ ^{o,\mathbb F}(I_{\Lbrack0,\tau\Rbrack})_t= P(\tau \ge t | {\cal F}_t).\end{eqnarray}
Furthermore, the following process $m$ given by 
\begin{equation} \label{processm}
m := G + D^{o,\mathbb F},\quad \mbox{where}\quad{D}:=I_{\Lbrack\tau,+\infty\Lbrack},
\end{equation}
is a BMO $\mathbb F$-martingale and plays important role in our analysis. The flow of information $\mathbb G$, which incorporates both $\mathbb F$ and $\tau$, is defined as follows. 
\begin{equation}\label{processD}
\mathbb G:=({\cal G}_t)_{t\geq 0},\ {\cal G}_t:={\cal G}^0_{t+}{{:=\bigcap_{s>t}{\cal G}^0_{s}}}\ \mbox{where} \ {\cal G}_t^0:={\cal F}_t\vee\sigma\left(D_s,\ s\leq t\right)\ .
\end{equation}
 Throughout the paper, on $\Omega\times [0,+\infty)$, we consider the $\mathbb F$-optional  $\sigma$-field  denoted by ${\cal O}(\mathbb F)$ and  the $\mathbb F$-progressive  $\sigma$-field denoted by $\mbox{Prog}(\mathbb F)$. Thanks to  \cite[Theorem 2.3]{Choulli1} and \cite[Theorem 2.3 and Theorem 2.11]{ChoulliDavelooseVanmaele}, we recall 
\begin{theorem}\label{Toperator} The following assertions hold.\\
{\rm{(a)}} For any $M\in{\cal M}_{loc}(\mathbb F)$,  the process
\begin{equation} \label{processMhat}
{\cal T}(M) := M^{\tau} -{\widetilde{G}}^{-1} I_{\Rbrack 0,\tau\Rbrack} \is [M,m] +  I_{\Rbrack 0,\tau\Rbrack} \is\Big(\sum \Delta M I_{\{\widetilde G=0<G_{-}\}}\Big)^{p,\mathbb F}\end{equation}
 is a $\mathbb G$-local martingale {{(recall that $G_{t-}$ coincides with $P(\tau \geq t | {\cal F}_{t-})=E[\widetilde{G}_t\big|{\cal{F}}_{t-}]$ for $t>0$). }}\\
 {\rm{(b)}}  The process 
\begin{equation} \label{processNG}
N^{\mathbb G}:=D - \widetilde{G}^{-1} I_{\Rbrack 0,\tau\Rbrack} \is D^{o,\mathbb  F}
\end{equation}
is a $\mathbb G$-martingale with integrable variation. Moreover, $H\is N^{\mathbb G}$ is a $\mathbb G$-local martingale with locally integrable variation for any $H$ belonging to ${\mathcal{I}}^o_{loc}(N^{\mathbb G},\mathbb G) $ given by 
\begin{equation} \label{SpaceLNG}
{\mathcal{I}}^o_{loc}(N^{\mathbb G},\mathbb G) := \Big\{K\in \mathcal{O}(\mathbb F):\quad \vert{K}\vert G{\widetilde G}^{-1} I_{\{\widetilde{G}>0\}}\is D\in{\cal A}_{loc}(\mathbb G)\Big\}.
\end{equation}
\end{theorem}
\noindent For any $q\in [1,+\infty)$ and a $\sigma$-algebra ${\cal H}$ on $\Omega\times [0,+\infty)$, we define
\begin{equation}\label{L1(PandD)Local}
L^q\left({\cal H}, P\otimes {{dD}}\right):=\left\{ X\ {\cal H}\mbox{-measurable}:\quad {E}[\vert X_{\tau}\vert^q I_{\{\tau<+\infty\}}]<+\infty\right\}.\end{equation}
Throughout the paper, we assume the following assumption 
\begin{eqnarray}\label{Assumption4Tau}
 G>0\quad (\mbox{i.e., $G$ is a positive process) and}\quad \tau<+\infty\quad P\mbox{-a.s.}.
\end{eqnarray}
Under the positivity of $G$, this process can be decomposed multiplicatively into two processes, which play central roles in the paper, as follows. 
\begin{lemma}\label{Decomposition4G} If $G>0$, then $\widetilde{G}>0$, $G_{-}>0$, and $G=G_0{\cal{E}}(G_{-}^{-1}\is m){\widetilde{\cal{E}}}$, where
\begin{equation}\label{EpsilonTilde}
\widetilde{\cal{E}}:={\cal E}\left(-{\widetilde{G}}^{-1}\is D^{o,\mathbb{F}}\right).\end{equation}
\end{lemma}
For more details about this lemma and related results, we refer the reader to \cite[Lemma 2.4]{Choulli1}. Below, we recall \cite[Proposition 4.3]{Choulli1} that will play important role throughout the paper.
\begin{proposition}Suppose that $G>0$ and consider the process
 \begin{equation}\label{Ztilde}
\widetilde{Z}:=1/{\cal E}(G_{-}^{-1}\is{m}).
\end{equation}
Then the following assertions hold.\\
{\rm{(a)}} $\widetilde{Z}^{\tau}$ is a $\mathbb G$-martingale, and for any $T\in (0,+\infty)$, $\widetilde{Q}_T$ given by 
 \begin{equation}\label{Qtilde}
 \frac{d{\widetilde{Q}_T}}{dP}:=\widetilde{Z}_{T\wedge\tau}.
\end{equation}
is well defined probability measure on ${\cal G}_{\tau\wedge T}$.\\
{\rm{(b)}} For any  $M\in {\cal M}_{loc}(\mathbb F)$, we have $M^{T\wedge \tau}\in {\cal M}_{loc}(\mathbb G, \widetilde{Q}_T)$.  In particular $W^{T\wedge\tau}$ is a Brownian motion for $(\widetilde{Q}_T, \mathbb G)$, for any $T\in (0,+\infty)$.
\end{proposition}
\begin{remark} {{1)} Under the condition $G>0$, we get ${\cal T}(M) := M^{\tau} -{\widetilde{G}}^{-1} I_{\Rbrack 0,\tau\Rbrack} \is [M,m]$ for any $M\in{\cal{M}}_{loc}(\mathbb{F})$. This due to the fact that when $G>0$, we get $\widetilde{G}>0$ thanks to Lemma \ref{Decomposition4G}. Thus, the process $I_{\{\widetilde{G}=0>G_{-}\}}$ is null, and as a consequence $ I_{\Rbrack 0,\tau\Rbrack} \is\Big(\sum \Delta M I_{\{\widetilde G=0<G_{-}\}}\Big)^{p,\mathbb F}=0$.\\
2)} In general, the $\mathbb G$-martingale $\widetilde{Z}^{\tau}$ might not be uniformly integrable, and hence in general $\widetilde{Q}_T$ might not be well defined for $T=\infty$. For these facts, we refer the reader to \cite[Proposition 4.3]{Choulli1}, where conditions for $\widetilde{Z}^{\tau}$  being uniformly integrable are fully singled out when $G>0$. 
\end{remark}
\section{The optimal stopping problem under random horizon}

Throughout the rest of the paper,  $\mathcal{J}_{\sigma_1}^{\sigma_2}(\mathbb{H})$ denotes the set of all $\mathbb{H}$-stopping times with values in $\Lbrack
\sigma_1,\sigma_2\Rbrack$, for any two $\mathbb H$-stopping times $\sigma_1$ and  $\sigma_2$  such that $\sigma_1\leq\sigma_2$ $P$-a.s.. This section has three subsections. The first subsection connects the $\mathbb{G}$-reward process to an $\mathbb{F}$-reward process in a unique manner, and investigates how their integrability is transmitted back and forth. The second subsection elaborates the mathematical structures results induced by $\tau$. The third subsection connects the minimal and maximal $\mathbb{G}$-optimal stopping times to the corresponding  $\mathbb{F}$-optimal stopping times.
\subsection{Parametrization  of $\mathbb{G}$-reward using $\mathbb{F}$-processes}
This subsection establishes exact relationship between $\mathbb{G}$-reward and $\mathbb{F}$-reward processes, and shows how some features travel back and forth between them. To this end, we recall the notion of class-D-processes.
\begin{definition}Let $(X,\mathbb{H})$ be a pair of a process $X$ and a filtration $\mathbb{H}$. Then $X$ is said to be of class-$(\mathbb{H},{\cal{D}})$ if $\{X_{\sigma}:\ \sigma\ \mbox{is a finite $\mathbb{H}$-stopping time}\}$ is a uniformly integrable family of random variables .  
\end{definition}
Below we state our main results of this subsection. {{To this end, throughout the rest of the paper, we consider the following notation
\begin{equation}
\sigma_1\wedge\sigma_2:=\min(\sigma_1,\sigma_2),\quad\mbox{for any two real-valued random variables $\sigma_1$ and $\sigma_2$}.
\end{equation}}}
\begin{theorem}\label{PropositionG2F}Suppose (\ref{Assumption4Tau}) holds, and let $X^{\mathbb G}$ be a $\mathbb G$-optional process such that $(X^{\mathbb G})^{\tau}=X^{\mathbb G}$. Then there exists a pair $(X^{\mathbb F}, k^{(pr)})$ of processes such that $X^{\mathbb F}$ is $\mathbb F$-optional, $k^{(pr)}$ is $\mathbb F$-progressive, 
\begin{eqnarray}\label{EqualityG2F}
X^{\mathbb G}=k^{(pr)}_0{I}_{\{\tau=0\}}+X^{\mathbb F} I_{\Lbrack0,\tau\Lbrack}+k^{(pr)}\is{D}\quad\mbox{and}\quad X^{\mathbb F} ={^{o,\mathbb{F}}}(X^{\mathbb G} I_{\Lbrack0,\tau\Lbrack})/G.
\end{eqnarray}
The pair $(X^{\mathbb F}, k^{(pr)})$  is unique in the sense that if there exists another pair $(\overline{X}^{\mathbb F}, \overline{k}^{(pr)})$ satisfying (\ref{EqualityG2F}), then $X^{\mathbb F}$ and $\overline{X}^{\mathbb F}$ are modifications of each other and $\overline{k}^{(pr)}={k}^{(pr)}$ $P\otimes{ {{dD}}}$-a.e..\\ Furthermore,  the following assertions hold.\\
{\rm{(a)}} $X^{\mathbb G}$ is locally bounded if and only if $X^{\mathbb F}$  and $ k^{(pr)}$ are is  locally bounded.\\
{\rm{(b)}}  $X^{\mathbb G}$ is RCLL if and only if $X^{\mathbb F}$ is RCLL.\\
{\rm{(c)}} $X^{\mathbb G}$ is a RCLL $\mathbb G$-semimartingale if and only if $X^{\mathbb F}$ is a RCLL $\mathbb F$-semimartingale. Furthermore, 
\begin{eqnarray}\label{Decompo4XG}
X^{\mathbb G}=(X^{\mathbb F})^{\tau}+(k^{(pr)}-X^{\mathbb F})\is{D}+(X^{\mathbb F}_0-k^{(pr)}_0)I_{\{\tau=0\}}.
\end{eqnarray}
{\rm{(d)}} For any $q\in [1,\infty)$, $E\left[\sup_{t\geq 0} \vert X^{\mathbb G}_{t}\vert^q\right]<\infty$ if and only if 
\begin{equation}
k^{(pr)}\in  L^q\left(\widetilde\Omega, {\rm{Prog}}(\mathbb F), P\otimes {{dD}}\right)\ \mbox{and}\ \sup_{0\leq s< \cdot}\vert X^{\mathbb F}_s\vert^q\is{D}^{o,\mathbb F}\in{\cal{A}}^+(\mathbb{F}).\end{equation}
{\rm{(e)}} If $X^{\mathbb G}$ is of class-$(\mathbb{G},{\cal{D}})$, then $k^{(pr)}\in  L^1\left(\widetilde\Omega, {\rm{Prog}}(\mathbb F), P\otimes {{dD}}\right)$ and $X^{\mathbb F}G$ is of class-$(\mathbb{F},{\cal{D}})$.
\end{theorem}
{{The local boundedness of the process $k^{(pr)}$, which is defined up to a $P\otimes{dD}$-evanescent set, is understood in the sense that there exist a sequence of $\mathbb{F}$-stopping times $(T_n)_{n\geq 1}$ which increases to infinity almost surely  and $\vert{k}^{(pr)}\vert{I}_{\Lbrack0,T_n\Rbrack}\leq C_n$, $P\otimes{dD}$-a.e. for all $n\geq 1$ for some $C_n\in (0,\infty)$, or equivalently $\vert{k}^{(pr)}_{\tau}\vert{I}_{\{\tau\leq{T}_n\}}\leq C_n$, $P$-a.s..}}
\begin{proof}[Proof of Theorem \ref{PropositionG2F}] Consider a $\mathbb G$-optional process $X^{\mathbb G}$. Then thanks to  \cite[Lemma B.1]{Aksamit} (see also  \cite[Lemma 4.4]{Jeulin1980}), there exists a pair $(X^{\mathbb F}, k^{(pr)})$ such that $X^{\mathbb F}$ is an $\mathbb F$-optional, $k^{(pr)}$ is $\mbox{Prog}(\mathbb F)$-measurable,  
\begin{equation}\label{Existence}
X^{\mathbb G}I_{\Lbrack0,\tau\Lbrack}=X^{\mathbb F}I_{\Lbrack0,\tau\Lbrack},\quad\mbox{and}\quad X^{\mathbb G}_{\tau}=k^{(pr)}_{\tau},\quad\quad P\mbox{-a.s.}.
\end{equation}
Furthermore, on the one hand, the uniqueness of this pair follows from the assumption $G>0$ and the second equality in (\ref{Existence}). On the other hand, $X^{\mathbb G}=(X^{\mathbb G})^{\tau}$ yields
\begin{equation*}
X^{\mathbb G}=X^{\mathbb G}I_{\Lbrack0,\tau\Lbrack}+X^{\mathbb G}_{\tau} I_{\Lbrack\tau,+\infty\Lbrack}=X^{\mathbb F}I_{\Lbrack0,\tau\Lbrack}+k^{(pr)}\is{D}+k^{(pr)}_0{I}_{\{\tau=0\}},
\end{equation*}
and the equality (\ref{EqualityG2F}) is proved. \\
a) In virtue of (\ref{Existence}), it is clear the local boundedness of the pair $(X^{\mathbb F}, k^{(pr)})$ implies the local boundedness of $X^{\mathbb G}$. To prove the reverse, we assume that $X^{\mathbb G}$ is locally bounded and $(T^{\mathbb{G}}_n)_n$ is the localizing sequence of stopping times. Hence, there exists a sequence of positive constants $(C_n)_n$ such that 
\begin{equation}\label{Bound4X(G)}
\esssup_{{t}\geq0}\vert{X}^{\mathbb G}_{{t}\wedge{T}^{\mathbb{G}}_n}\vert\leq{C}_n,\quad\quad P\mbox{-a.s..}
\end{equation}
 Then there exists a sequence of $\mathbb{F}$-stopping times $(T_n)_n$ that increases to infinity almost surely and $\min(\tau,T_n)=\min(T^{\mathbb{G}}_n,\tau)$ $P$-a.s. In virtue of the assumption $(X^{\mathbb G})^{\tau}=X^{\mathbb G}$ and (\ref{Existence}), the inequality (\ref{Bound4X(G)}) is equivalent to 
 \begin{equation*}
  \max\left( \esssup_{0\leq{t}<\tau}\vert{X}^{\mathbb F}_{{t}\wedge{T}_n}\vert,\vert{X}^{\mathbb G}_{\tau\wedge{T}_n}\vert\right)=\max\left( \esssup_{0\leq{t}<\tau}\vert{X}^{\mathbb G}_{{t}\wedge{T}^{\mathbb{G}}_n}\vert,\vert{X}^{\mathbb G}_{\tau\wedge{T}^{\mathbb{G}}_n}\vert\right)\leq C_n.
 \end{equation*}
Hence this inequality implies $\vert{k}^{(pr)}_{\tau}\vert{I}_{\{\tau\leq{T}_n\}}\leq \vert{X}^{\mathbb G}_{\tau\wedge{T}_n}\vert\leq C_n$ and $ \esssup_{0\leq{t}<\tau}\vert{X}^{\mathbb F}_{{t}\wedge{T}_n}\vert\leq{C}_n$ $P$-a.s., or equivalently $ {k}^{(pr)}I_{\Lbrack0,T_n\Rbrack}$ and $({X}^{\mathbb F})^{T_n}I_{\Lbrack0,\tau\Lbrack}$ are bounded by $C_n$. Thanks to $G>0$ and the fact that $T_n$  is an $\mathbb{F}$-stopping time that increases to infinity, these latter conditions are equivalent to say that both ${k}^{(pr)}$ and ${X}^{\mathbb F}$ are locally bounded. This proves assertion (a).\\
b) Thanks to (\ref{EqualityG2F}) and the fact that $k^{(pr)}\is{D}$ is a RCLL process, we deduce that $X^{\mathbb G}$ is RCLL if and only if $X^{\mathbb F} I_{\Lbrack0,\tau\Lbrack}$ is RCLL. Thus, we suppose $X^{\mathbb G}$ being a RCLL process, and we consider the sequence of $\mathbb G$-stopping times  $(T^{\mathbb G}_n)$ given by 
\begin{eqnarray*}
T_n^{\mathbb{G}}:=\inf\left\{t\geq 0:\ \vert{X}^{\mathbb G}_t\vert>n\right\},\ \mbox{and satisfies}\  \vert{X}^{\mathbb G,n}\vert\leq n,\ X^{\mathbb G,n}:= X^{\mathbb G}{I}_{\Lbrack0,T_n^{\mathbb{G}}\Lbrack}.
\end{eqnarray*}
It is clear that $T_n^{\mathbb{G}}$ increases to infinity, and by virtue of  \cite[Proposition B.2-(b)]{Aksamit} and $G>0$, there exists a sequence of $\mathbb F$-stopping times $(T_n)_n$ which increases to infinity  and satisfies $T_n^{\mathbb{G}}\wedge\tau=T_n\wedge\tau$. Furthermore, by applying (\ref{EqualityG2F})  to each ${X}^{\mathbb G,n}$, on the one hand, we get
 \begin{eqnarray*}
 {X}^{\mathbb G,n}I_{\Lbrack0,\tau\Lbrack}=X^{\mathbb F}I_{\Lbrack0,\tau\Lbrack}I_{\Lbrack0,T_n\Lbrack}.
 \end{eqnarray*}
 On the other hand, as $T_n$ increases to infinity, it is clear that $ X^{\mathbb F}$ is a RCLL if and only if $X^{\mathbb F}I_{\Lbrack0,T_n\Lbrack}={^{o,\mathbb{F}}}( {X}^{\mathbb G,n}I_{\Lbrack0,\tau\Lbrack})G^{-1}$ is RCLL. This latter fact follows directly from combining the boundedness of $ {X}^{\mathbb G,n}$, \cite[Th\'eor\`eme 47, pp: 119]{DellacherieMeyer80}, and the right-continuity of $G$. This completes the proof of assertion (b).\\
c) It is clear that $k^{(pr)}\is{ D}$ is a RCLL $\mathbb G$-semimartingale, and hence $X^{\mathbb G}$ is a RCLL $\mathbb G$-semimartingale if and only if $X^{\mathbb F}I_{\Lbrack0,\tau\Lbrack}$ is a RCLL $\mathbb G$-semimartingale. Thus, if $X^{\mathbb{F}}$ is a RCLL $\mathbb F$-semimartingale, then $X^{\mathbb G}$ is a RCLL $\mathbb G$-semimartingale. To prove the converse, we remark that by stopping with $T^{\mathbb G}_n$ defined above and by using \cite[Th\'eor\`eme 26, Chapter VII, pp: 235]{DellacherieMeyer80}, there is no loss of generality in assuming  $X^{\mathbb G}$ is bounded, which leads to the boundedness of $X^{\mathbb F}$, see  \cite[Lemma B.1]{Aksamit}  or  \cite[Lemma 4.4 (b), pp: 63]{Jeulin1980}. Thus, thanks to \cite[Th\'eor\`eme 47, pp: 119 and Th\'eor\`eme 59, pp: 268]{DellacherieMeyer80} which implies that the optional projection of a bounded RCLL $\mathbb G$-semimartingale is a RCLL $\mathbb F$-semimartingale, we deduce that $ X^{\mathbb F}G=^{o,\mathbb F}\left( X^{\mathbb{G}}I_{\Lbrack0,\tau\Lbrack}\right)$ is a RCLL $\mathbb F$-semimartingale. A combination of this with the condition $G>0$ and the fact that $G$ is a RCLL $\mathbb F$-semimartingale implies that $ X^{\mathbb F}$ is a RCLL $\mathbb F$-semimartingale. Furthermore, direct calculation yields 
 \begin{eqnarray*}\label{equalityG2Fbis}
 X^{\mathbb F}I_{\Lbrack0,\tau\Lbrack}=( X^{\mathbb F})^{\tau}- X^{\mathbb F}\is{D}-X^{\mathbb F}_0{I}_{\{\tau=0\}}\quad\mbox{is a $\mathbb G$-semimartingale},\end{eqnarray*}
 and (\ref{Decompo4XG}) follows from this equality and (\ref{EqualityG2F}). \\
d)  Here, we prove assertion (d). To this end, we use  (\ref{EqualityG2F}) and derive
\begin{eqnarray*}
{{I_{\infty}}\over{2}} \leq \sup_{t\geq 0} \vert X^{\mathbb G}_{t}\vert^q=\max\left(\sup_{ 0\leq{t}<\tau} \vert X^{\mathbb F}_{t}\vert^q,\vert{k}^{(pr)}_{\tau}\vert^q\right)\leq{I}_{\infty},\end{eqnarray*}
where $ I:=\left(\sup_{ 0\leq{u}<\cdot} \vert X^{\mathbb F}_{u}\vert^q+\vert{k}^{(pr)}\vert^q\right)\is{D}.$ Hence, $E\left[\sup_{t\geq 0} \vert X^{\mathbb G}_{t}\vert^q\right]<\infty$ if and only if boh $E\left[\int_0^{\infty} \vert{k}^{(pr)}_t\vert^q{d}D_t\right]$ and 
$E\left[\int_0^{\infty} \sup_{ 0\leq{u}<t} \vert{X}^{\mathbb F}_{u}\vert^q{d}D_t^{o,\mathbb F}\right]$
 are finite. This proves assertion (d).\\
e) Suppose that  $X^{\mathbb G}$ is of class-$(\mathbb{G},{\cal{D}})$. On the one hand, we have  $E[\vert{k}^{(pr)}_{\tau}\vert]=E[\vert{X}^{\mathbb G}_{\tau}\vert]<\infty$, or equivalently $k^{(pr)}\in  L^1\left(\widetilde\Omega, {\rm{Prog}}(\mathbb F), P\otimes {{dD}}\right)$. On the other hand, due to $G_{\sigma}\leq 1$, for any $c>0$, we have 
\begin{align*}
E\left[\vert{X}^{\mathbb F}_{\sigma}\vert{G}_{\sigma}I_{\{\vert{X}^{\mathbb F}_{\sigma}\vert{G}_{\sigma}>c\}}\right]&\leq E\left[\vert{X}^{\mathbb F}_{\sigma}\vert{G}_{\sigma}I_{\{\vert{X}^{\mathbb F}_{\sigma}\vert>c\}}\right]=E\left[\vert{X}^{\mathbb F}_{\sigma}\vert{I}_{\{\sigma<\tau\}}I_{\{\vert{X}^{\mathbb F}_{\sigma}\vert>c\}}\right]\\
&=E\left[\vert{X}^{\mathbb G}_{\sigma}\vert{I}_{\{\sigma<\tau\}}I_{\{\vert{X}^{\mathbb G}_{\sigma}\vert>c\}}\right].\end{align*}
This proves that $X^{\mathbb F}G$ is of class-$(\mathbb{F},{\cal{D}})$, and the proposition is proved. 
\end{proof}
\subsection{The mathematical structures of the value process (Snell envelop)}
The main result of this subsection characterizes, in different manners, the Snell envelop of a process under $\mathbb G$ in terms of $\mathbb F$-Snell envelop and some $\mathbb G$-local martingales. To this end, we start with the following two lemmas.
\begin{lemma} \label{G-projection}For any nonnegative or integrable process $X$, we always have 
\begin{eqnarray*}\label{converting}E\left[X_{t}|\mathcal{G}_{t}\right]I_{\{t\ <\tau\}}={E\left[X_{t}I_{\{t\ <\tau\}}|\mathcal{F}_{t}\right]}G_t^{-1}I_{\{t\ <\tau\}}.
\end{eqnarray*}
\end{lemma}
This follows from \cite[XX.75-(c)/(d)]{DMM}, while the next lemma sounds new. 
\begin{lemma}\label{stoppingTimeLemma}
Let $\sigma_{1}$ and $\sigma_{2}$ be two $\mathbb{F}$-stopping times such that $\sigma_{1}\leq\sigma_{2}$ P-a.s.. Then, for any $\mathbb{G}$- stopping time, $\sigma^{\mathbb{G}}$, satisfying
\begin{equation}\label{sigmaG}
\sigma_{1}\wedge\tau\leq\sigma^{\mathbb{G}}\leq\sigma_{2}\wedge\tau \hspace{5mm} P\mbox{-a.s.},
\end{equation}
there exists an $\mathbb{F}$-stopping time $\sigma^{\mathbb{F}}$ such that 
\begin{equation}\label{sigmaF}
\sigma_{1}\leq\sigma^{\mathbb{F}}\leq\sigma_{2}\quad \mbox{ and }\quad \sigma^{\mathbb{F}}\wedge\tau=\sigma^{\mathbb{G}} \hspace{5mm} P\mbox{-a.s.}
\end{equation}
\end{lemma}
 The proof of this lemma is relegated to Appendix \ref{Appendix4Proofs}. Throughout the paper, for any $\cal{F}\times\cal{B}(\mathbb R^+)$-measurable process $X$ --which is nonnegative or $\mu:=P\otimes {{dD}}$-integrable--, its $\mathbb F$-optional projection with respect to $\mu$, denoted by $M_{\mu}^P(X\big|{\cal O}(\mathbb F))$, is the unique $\mathbb F$-optional process $Y$ such that for any bounded and $\mathbb F$-optional $H$, 
\begin{equation}\label{OptionalProjection4mu}
M_{\mu}^P[XH]:=E\left[\int_0^{\infty} X_sH_s dD_s\right]=E\left[\int_0^{\infty} Y_sH_s dD_s\right].
\end{equation}
\begin{theorem}\label{SnellEvelopG2F}Suppose $G>0$, and let $X^{\mathbb G}$ be a RCLL and $\mathbb G$-adapted process such that $(X^{\mathbb G})^{\tau}=X^{\mathbb G}$ and $\sup_{0\leq s\leq\cdot}\vert{X}^{\mathbb G}_s\vert\in{\cal{A}}^+_{loc}(\mathbb{G})$. Then consider the unique pair $(X^{\mathbb F}, k^{(pr)})$ associated to $X^{\mathbb G}$ via Theorem \ref{PropositionG2F}, and denote by 
\begin{equation}\label{k(F)}
k^{(op)}:=M^P_{\mu}( k^{(pr)}\big|{\cal O}(\mathbb F))\ \rm{and}\ k^{(\mathbb{F})}:=k^{(pr)}-k^{(op)},\end{equation}
where $\mu:=P\otimes {{dD}}$ see (\ref{OptionalProjection4mu}). Then the following assertions hold.\\
{\rm{(a)}} If either $X^{\mathbb G}$ is nonnegative or $E\left[\sup_{t\geq 0} (X^{\mathbb G}_{t})^+\right]<\infty$, then the $(\mathbb G,P)$-Snell envelop of $X^{\mathbb G}$, denoted ${\cal S}^{\mathbb G}$, is given by
\begin{equation}\label{Snell4(G,P)}
{\cal S}^{\mathbb G}={{{\cal S}^{\mathbb F}}\over{G}}I_{\Lbrack0,\tau\Lbrack}+{{(k^{(op)}\is{ D}^{o,\mathbb F})_{-}}\over{G_{-}^2}}\is{\cal T}(m)+k^{(\mathbb{F})}\is{ D}+\left(k^{(op)}+{{k^{(op)}\is{D}^{o,\mathbb F}}\over{G}}\right)\is{ N}^{\mathbb G},
\end{equation}
where ${\cal S}^{\mathbb F}$ is the $(\mathbb F, P)$-Snell envelop of the reward $\widetilde{X}^{\mathbb F}:=X^{\mathbb F}G+k^{(op)}\is{ D}^{o,\mathbb F}$.\\
{\rm{(b)}} Let T$\in (0,+\infty)$ and $\widetilde{Q}$ given in (\ref{Qtilde}).  If either $\displaystyle{E}^{\widetilde{Q}}\left[\sup_{0\leq t\leq{T}} (X^{\mathbb G}_{t})^+\right]<\infty$ or $X^{\mathbb G}\geq 0$, then the $(\mathbb G, \widetilde{Q})$-Snell envelop of $(X^{\mathbb G})^T$, denoted ${\cal S}^{\mathbb G,\widetilde{Q}}$, satisfies
\begin{eqnarray}\label{Snell4(G,Qtilde)}
{\cal S}^{\mathbb G,\widetilde{Q}}={{\widetilde{\cal S}^{\mathbb F}}\over{\widetilde{\cal E}^T}}(I_{\Lbrack0,\tau\Lbrack})^T+k^{(\mathbb{F})}\is{D}^T
+\left(k^{(op)}-{{k^{(op)}\is \widetilde{\cal E}}\over{\widetilde{\cal E}}}\right)\is (N^{\mathbb G})^T.
\end{eqnarray}
Here $\widetilde{\cal S}^{\mathbb F}$ is $(\mathbb F, P)$-Snell envelop of  the $\mathbb{F}$-reward $(X^{\mathbb F}{\widetilde{\cal E}}-k^{(op)}\is{\widetilde{\cal E}})^T$.
\end{theorem}
\begin{remark}\label{Remarks4TheoremXXXX}
(a) The condition  $\sup_{0\leq s\leq\cdot}\vert{X}^{\mathbb G}_s\vert\in{\cal{A}}^+_{loc}(\mathbb{G})$ implies that the two last terms in the right-hand-side of (\ref{Snell4(G,P)}) and (\ref{Snell4(G,Qtilde)}) are well defined $\mathbb{G}$-local martingales. In fact,  in virtue of Theorem \ref{PropositionG2F}-(c), $\sup_{0\leq s\leq\cdot}\vert{X}^{\mathbb G}_s\vert\in{\cal{A}}^+_{loc}(\mathbb{G})$ yields $k^{(pr)}\in L^1_{loc}\left(\widetilde\Omega, {\rm{Prog}}(\mathbb F), P\otimes {{dD}}\right)$, or equivalently the pair $(k^{(\mathbb{F})},k^{(op)})$ belongs to $L^1_{loc}\left(\widetilde\Omega, {\rm{Prog}}(\mathbb F), P\otimes {{dD}}\right)\times  L^1_{loc}\left(\widetilde\Omega, {\cal{O}}(\mathbb F), P\otimes {{dD}}\right)$. On the one hand, this condition obviously implies that $k^{(\mathbb{F})}\is{D}\in{\cal{M}}_{loc}(\mathbb{G})$ and $k^{(op)}$ belongs to ${\mathcal{I}}^o_{loc}(N^{\mathbb G},\mathbb G)$. On the other hand, by considering $$\sigma_n:=\inf\{t\geq 0:\ \widetilde{\cal{E}}_t<n^{-1}\quad\mbox{or}\quad \vert{k}^{(op)}\is{D}^{o,\mathbb{F}}_t\vert>n\}\quad\mbox{ and}\quad{T}_n:=n\wedge\sigma_n,$$ which both increase to infinity,  and using 
$$G^{-1}=G_0^{-1}{\cal{E}}(G_{-}^{-1}\is{m})^{-1}{\widetilde{\cal{E}}}^{-1}=G_0^{-1}{{\cal{E}}(G_{-}^{-1}\is{m})^{-1}}{\widetilde{\cal{E}}}^{-1}_{-}\widetilde{G}/G,$$
 we get 
  \begin{align*} 
  E\left[\int_0^{T_n}{{G_t}\over{\widetilde{G}_t}}{{ \vert{k}^{(op)}\is{D}^{o,\mathbb{F}}_t\vert}\over{G_t}}dD_t  \right]&\leq {G}_0^{-1}n^2+ E\left[\int_0^{T_n}{{G_t}\over{\widetilde{G}_t}}{{ \vert{k}^{(op)}_t\vert\Delta{D}^{o,\mathbb{F}}_t}\over{G_t}}dD_t  \right]\\
  &\leq  {G}_0^{-1}n^2+E\left[\int_0^{T_n} \vert{k}^{(op)}_t\vert{d}D_t  \right]<\infty.
  \end{align*}
  This proves that $G^{-1}({k}^{(op)}\is{D}^{o,\mathbb{F}})\in {\mathcal{I}}^o_{loc}(N^{\mathbb G},\mathbb G)$, and similar reasoning proves that the process ${\widetilde{\cal{E}}}^{-1}({k}^{(op)}\is{\widetilde{\cal{E}}})$ belongs to $ {\mathcal{I}}^o_{loc}(N^{\mathbb G},\mathbb G)$ also. Hence, the claim is proved.\\
(b) For $\theta\in {\cal T}_{t\wedge\tau}^{\tau}(\mathbb G)$ and $\sigma\in {\cal T}_t^{\infty}(\mathbb F)$ with  $\theta=\sigma\wedge\tau$, see Lemma \ref{stoppingTimeLemma}, we have 
\begin{eqnarray}\label{equa100}
\begin{split}
X_{\theta}^{\mathbb G}&=X_{\sigma\wedge\tau}^{\mathbb G}I_{\{\sigma<\tau\}}+ k^{(pr)}_{\tau}I_{\{\sigma\geq\tau\}}=X_{\sigma}^{\mathbb F}I_{\{\sigma<\tau\}}+\int_0^{\sigma} k^{(pr)}_s dD_s+ k^{(pr)}_0{I}_{\{\tau=0\}}\\
&=X_{\sigma}^{\mathbb F}I_{\{\sigma<\tau\}}+ ({{k^{(op)}}\over{\widetilde G}} \is{ D}^{o,\mathbb F})_{\sigma\wedge\tau}+ (k^{(op)} \is{  N}^{\mathbb G})_{\theta}+k^{(\mathbb{F})}\is{ D}_{\theta}+ k^{(pr)}_0{I}_{\{\tau=0\}},
\end{split}\end{eqnarray}
\end{remark}
This latter remark plays important role in proving Theorem  \ref{SnellEvelopG2F}. The rest of this subsection is devoted to this proof, and hence we start with the next lemma which is useful here and in the rest of the paper as well.
\begin{lemma}\label{L/EpsilonTilde}  Suppose $G>0$ and let ${\widetilde{\cal E}}$ be defined in (\ref{EpsilonTilde}). Then the following assertions hold.\\
{\rm{(a)}} For any RCLL $\mathbb F$-semimartingale $L$, it holds that 
\begin{equation}\label{LoverEpsilonTilde}
L\widetilde{\cal E}^{-1}I_{\Lbrack0,\tau\Lbrack}+L\widetilde{\cal E}^{-1}\is N^{\mathbb G}=L_0I_{\{\tau>0\}}+\widetilde{\cal E}_{-}^{-1}\is L^{\tau},
\end{equation}
and 
\begin{equation}\label{LoverG}
{{L}\over{G}}I_{\Lbrack0,\tau\Lbrack}={{L_0}\over{G_0}}I_{\{\tau>0\}}-{{L_{-}}\over{G_{-}^2}}I_{\Rbrack0,\tau\Rbrack}\is {\cal{T}}(m)+{1\over{G_{-}}}I_{\Rbrack0,\tau\Rbrack}\is{\cal{T}}(L)- {{L}\over{G}}I_{\Rbrack0,\tau\Rbrack}\is N^{\mathbb G}.
\end{equation}
{\rm{(b)}} For any $\mathbb{F}$-optional process $k$ such that $V:=k\is D^{o,\mathbb{F}}\in {\cal A}_{loc}(\mathbb{F})$, we have 
\begin{equation}\label{V/G}
-{{V^{\tau}}\over{G^{\tau}}}={{V_{-}}\over{G_{-}^2}}\is\underbrace{(m^{\tau}-{\widetilde{G}}^{-1}\is [m,m]^{\tau})}_{={\cal{T}}(m)}-{{k+VG^{-1}}\over{\widetilde{G}}}I_{\Rbrack0,\tau\Rbrack}\is D^{o,\mathbb{F}}.
\end{equation}
\end{lemma}
The proof is given in Appendix \ref{Appendix4Proofs}, while below we prove Theorem \ref{SnellEvelopG2F}.
\begin{proof}[Proof of Theorem \ref{SnellEvelopG2F}] This proof is divided into three parts. The first and second parts prove assertions (a) and (b) respectively when $X^{\mathbb{G}}$ is bounded, while the third part relaxes this condition and proves the theorem.\\
{\bf Part 1.} In this part, we suppose $X^{\mathbb G}$ is bounded, and proves assertion (a). Hence, under this assumption, the associated processes $X^{\mathbb F}, k^{(pr)}$ and $k^{(op)}$ are also bounded. As a result, both $ k^{(op)} \is{ N}^{\mathbb G}$ and $ k^{(\mathbb{F})}\is{ D}$ are uniformly integrable $\mathbb G$-martingales. Thus, by putting
\begin{eqnarray}\label{LGprocess}
L^{\mathbb{G}}:=k^{(op)} \is{ N}^{\mathbb G}+k^{(\mathbb{F})}\is{ D},
\end{eqnarray}
combining the remarks above with Lemma \ref{G-projection}, and taking conditional expectation with respect to ${\cal G}_t$ on both sides of (\ref{equa100}), we derive 
\begin{align}
Y_t(\theta)&:= E\left[ X_{\theta}^{\mathbb G}\big|{\cal G}_t\right]= E\left[ X_{\sigma}^{\mathbb F}I_{\{\sigma<\tau\}}+\int_0^{\sigma\wedge\tau} {{k^{(op)}_s}\over{{\widetilde G}_s}}  dD^{o,\mathbb F}_s \big|{\cal G}_t\right]+L^{\mathbb{G}}_t \nonumber\\
&=E\left[ X_{\sigma}^{\mathbb F}I_{\{\sigma<\tau\}}+\int_{t\wedge\tau}^{\sigma\wedge\tau} {{k^{(op)}_s}\over{{\widetilde G}_s}} dD^{o,\mathbb F}_s \bigg|{\cal G}_t\right]+({{k^{(op)}}\over{{\widetilde G}}}\is{ D}^{o,\mathbb F})_{t\wedge\tau}+L^{\mathbb{G}}_t \nonumber\\
&=E\left[ X_{\sigma}^{\mathbb F}I_{\{\sigma<\tau\}}+\int_{t\wedge\tau}^{\sigma\wedge\tau} {{k^{(op)}_s}\over{{\widetilde G}_s}} dD^{o,\mathbb F}_s \bigg|{\cal F}_t\right]{{I_{\{\tau>t\}}}\over{G_t}}+{{k^{(op)}}\over{{\widetilde G}}}\is{ D}^{o,\mathbb F}_{t\wedge\tau}+L^{\mathbb{G}}_t\nonumber \\
&=E\left[ G_{\sigma}X_{\sigma}^{\mathbb F} +\int_t^{\sigma}k^{(op)}_s dD^{o,\mathbb F}_s \bigg|{\cal F}_t\right]{{I_{\{\tau>t\}}}\over{G_t}}+({{k^{(op)}}\over{{\widetilde G}}}\is{ D}^{o,\mathbb F})_{t\wedge\tau}+L^{\mathbb{G}}_t \nonumber\\
&=:{{X^{\mathbb F}_t(\sigma)-(k^{(op)}\is{ D}^{o,\mathbb F})_t}\over{G_t}} I_{\{t<\tau\}}+({{k^{(op)}}\over{{\widetilde G}}}\is{ D}^{o,\mathbb F})_{t\wedge\tau}+L^{\mathbb{G}}_t.\label{Equation4Theta}
\end{align}
Thus, by taking the essential supremum  and using  (\ref{LGprocess}), we get 
\begin{equation*}\label{mainequality1}
{\cal S}^{\mathbb G}={{{\cal S}^{\mathbb F}-(k^{(op)}\is{ D}^{o,\mathbb F})}\over{G}} I_{\Lbrack0,\tau\Lbrack}+{{k^{(op)}}\over{{\widetilde G}}}\is({ D}^{o,\mathbb F})^{\tau}+k^{(op)} \is{ N}^{\mathbb G}+k^{(\mathbb{F})}\is{ D} .
\end{equation*}
Therefore, by combining this with (\ref{V/G}) (see Lemma \ref{L/EpsilonTilde}-(c)), and 
\begin{eqnarray}\label{X-Fsemimartinagle}
XI_{\Lbrack0,\tau\Lbrack}=X^{\tau}-X\is{ D}-X_0{I}_{\{\tau=0\}},\quad \mbox{for any}\quad \mathbb F\mbox{-semimartingale}\ X,\end{eqnarray}
we obtain immediately (\ref{Snell4(G,P)}), and part 1 is completed.\\
{\bf Part 2.} Here, we suppose that $X^{\mathbb {G}}$ is bounded, we fix $T\in (0,+\infty)$ and prove assertion (b).  Let $\theta\in {\cal T}_{t\wedge\tau}^{T\wedge\tau}(\mathbb G)$ and $\sigma\in  {\cal T}_t^T(\mathbb F)$ such that $\theta=\sigma\wedge\tau$. Then, similarly as in Part 1, by taking  $\widetilde{Q}$-conditional expectation in both sides of (\ref{equa100}), and using (\ref{LGprocess}) and the fact that the two processes $ k^{(op)} \is{ N}^{\mathbb G}$ and $ k^{(\mathbb{F})}\is{ D}$ remain uniformly integrable $\mathbb G$-martingales under $\widetilde{Q}$ (due the boundedness of $k^{(pr)}$ and $k^{(\mathbb{F})}$), we write  
\begin{align}
 &{\widetilde Y}_t(\theta):=E^{\widetilde Q} \left[ X_{\theta}^{\mathbb G}\big|{\cal G}_t\right]=E^{\widetilde Q}\left[ X_{\sigma}^{\mathbb F}I_{\{\sigma<\tau\}}+\int_{0}^{\sigma\wedge\tau} {{k^{(op)}_s}\over{{\widetilde G}_s}} dD^{o,\mathbb F}_s \bigg|{\cal G}_t\right]+L^{\mathbb{G}}_t  \nonumber\\
 &=E\left[ {{\widetilde{Z}_{\sigma}}\over{\widetilde{Z}_t}}X_{\sigma}^{\mathbb F}I_{\{\sigma<\tau\}}+\int_{t\wedge\tau}^{\sigma\wedge\tau} {{k^{(op)}_s\widetilde{Z}_s}\over{{\widetilde G}_s\widetilde{Z}_t}} dD^{o,\mathbb F}_s \bigg|{\cal G}_t\right]+{{k^{(op)}}\over{{\widetilde G}}}\is{ D}^{o,\mathbb F}_{t\wedge\tau}+L^{\mathbb{G}}_t \nonumber \\
 &=E\left[ \widetilde{Z}_{\sigma}X_{\sigma}^{\mathbb F}I_{\{\sigma<\tau\}}+\int_{t\wedge\tau}^{\sigma\wedge\tau}G_0 {{k^{(op)}_s}\over{{\widetilde G}_s}}dV^{\mathbb F}_s \bigg|{\cal F}_t\right]{{I_{\{\tau>t\}}}\over{\widetilde{Z}_tG_t}}+{{k^{(op)}}\over{{\widetilde G}}}\is{ D}^{o,\mathbb F}_{t\wedge\tau}+L^{\mathbb{G}}_t ,\label{equation400}\end{align}
 where $V^{\mathbb F}$ is a RCLL and nondecreasing process given by 
\begin{equation}\label{V(F)process}
V^{\mathbb F}:=1-\widetilde{\cal{E}}\quad\mbox{or equivalently}\quad dV^{\mathbb F}={\widetilde{G}}^{-1}{\widetilde{\cal{E}}}_{-}d{ D}^{o,\mathbb F}=G_0^{-1}{\widetilde{Z}}d{ D}^{o,\mathbb F}\ \mbox{and}\  V^{\mathbb F}_0=0.
\end{equation}
 Thus, by simplifying more in the right-hand-side of (\ref{equation400}), we obtain 
 \begin{align}
 &{\widetilde Y}_t(\theta)=E\left[ \widetilde{\cal E}_{\sigma}X_{\sigma}^{\mathbb F}+\int_t^{\sigma} k^{(op)}_sdV^{\mathbb F}_s \Big|{\cal F}_t\right]{{I_{\{\tau>t\}}}\over{\widetilde{\cal E}_t}}+({{k^{(op)}}\over{{\widetilde G}}}\is{ D}^{o,\mathbb F})_{t\wedge\tau}+L^{\mathbb{G}}_t \nonumber\\
& =:{{X^{\mathbb F}_t(\sigma)}\over{\widetilde{\cal E}_t}}I_{\{\tau>t\}}-{{ (k^{(op)}\is{ V}^{\mathbb F})_t}\over{\widetilde{\cal E}_t}}I_{\{\tau>t\}}+({{k^{(op)}}\over{{\widetilde G}}}\is{ D}^{o,\mathbb F})_{t\wedge\tau}+L^{\mathbb{G}}_t. \label{Equation4ThetaQtilde}
\end{align}
By taking essential supremum over all $\theta\in {\cal T}_{t\wedge\tau}^{T\wedge\tau}(\mathbb G)$, we get 
\begin{equation*}\label{mainequality2}
{\cal S}^{\mathbb G,\widetilde{Q}}={{{\widetilde{\cal S}}^{\mathbb{F}}-k^{(op)}\is(V^{\mathbb F})^T }\over{\widetilde{\cal E}^T}}(I_{\Lbrack0,\tau\Lbrack})^T+({{k^{(op)}}\over{{\widetilde G}}}\is{ D}^{o,\mathbb F})^{T\wedge\tau}+(L^{\mathbb{G}})^T,
\end{equation*}
where ${\widetilde{\cal S}}^{\mathbb{F}}$ is the Snell envelop, for the reward $(X^{\mathbb F}{\widetilde{\cal E}}+k^{(op)}\is{ D}^{o,\mathbb F})^T$, under $(\mathbb F, P)$. Thus, by combining the above equality with (\ref{LoverEpsilonTilde}) (see Lemma \ref{L/EpsilonTilde}-(a)) and (\ref{LGprocess}), we get (\ref{Snell4(G,Qtilde)}). This ends the second part.\\
{\bf Part 3:} Herein, we prove the theorem without the boundedness assumption on $X^{\mathbb{G}}$. To this end, by virtue of parts 1 and 2, we remark that the theorem follows immediately as soon as we prove that the equalities (\ref{Equation4Theta}) and (\ref{Equation4ThetaQtilde}) hold. Thus, we consider for $n\geq 0$, 
\begin{eqnarray*}
X^{\mathbb{G},n}:=X^{\mathbb{G}}I_{\{\vert{X}^{\mathbb{G}}\vert\leq n\}},
\end{eqnarray*}
 and its associated triplet $(X^{\mathbb{F},n}, k^{(pr,n)},k^{(op,n)})$ is given by 
 \begin{equation*}
 X^{\mathbb{F},n}:=X^{\mathbb{F}}I_{\{\vert{X}^{\mathbb{F}}\vert\leq n\}},\quad  k^{(pr,n)}:=k^{(pr)}I_{\{\vert {k}^{(pr)}\vert\leq n\}},\quad k^{(op,n)}:=M^P_{\mu}\left(k^{(pr,n)}\big|{\cal{O}}(\mathbb{F})\right).\end{equation*}
 Therefore it is clear that $X^{\mathbb{G},n}$ and its associated triplet $(X^{\mathbb{F},n}, k^{(pr,n)},k^{(op,n)})$ are bounded, and thanks to  parts 1 and 2, we conclude that they fulfills (\ref{Equation4Theta}) and (\ref{Equation4ThetaQtilde}). If $X^{\mathbb G}\geq0$, then $X^{\mathbb{G},n}$ is nonnegative and increases to $X^{\mathbb{G}}$, and all components of $(X^{\mathbb{F},n}, k^{(pr,n)},k^{(op,n)})$ are nonnegative and increase to the corresponding components of $(X^{\mathbb{F}}, k^{(pr)},k^{(op)})$ respectively. Thus, thanks to the convergence monotone theorem, it is clear that in this case $E[X^{\mathbb{G},n}_{\theta}|{\cal{G}}_t]$ (respectively $E^{\widetilde{Q}}[X^{\mathbb{G},n}_{\theta}|{\cal{G}}_t]$) increases to $E[X^{\mathbb{G}}_{\theta}|{\cal{G}}_t]$ (respectively $E^{\widetilde{Q}}[X^{\mathbb{G}}_{\theta}|{\cal{G}}_t]$) and $E\left[ G_{\sigma}X_{\sigma}^{\mathbb{F},n}+\int_t^{\sigma} k^{(op,n)}_sdD^{o,\mathbb F}_s \big|{\cal F}_t\right]$ (respectively $E\left[ \widetilde{\cal E}_{\sigma}X_{\sigma}^{\mathbb{F},n}+\int_t^{\sigma} k^{(op,n)}_sdV^{\mathbb F}_s \big|{\cal F}_t\right]$) increases to $E\left[ G_{\sigma}X_{\sigma}^{\mathbb{F}}+\int_t^{\sigma} k^{(op)}_sdV^{\mathbb F}_s \big|{\cal F}_t\right]$ (respectively $E\left[ \widetilde{\cal E}_{\sigma}X_{\sigma}^{\mathbb{F},n}+\int_t^{\sigma} k^{(op,n)}_sdV^{\mathbb F}_s \big|{\cal F}_t\right]$). This proves that (\ref{Equation4Theta}) and (\ref{Equation4ThetaQtilde})  hold for the case when $X^{\mathbb G}\geq0$, and the theorem is proved in this case.  Now suppose that $E\left[\sup_{t\geq 0} (X^{\mathbb G}_{t})^+\right]<\infty$. Then Fatou's lemma yields 
\begin{equation}\label{Fatou}
 E[X^{\mathbb{G}}_{\theta}|{\cal{G}}_t]\geq \limsup_{n\longrightarrow+\infty} E[X^{\mathbb{G},n}_{\theta}|{\cal{G}}_t],
\end{equation}
while a combination of the convergence dominated theorem with the inequality  
\begin{equation*}
E[X^{\mathbb{G},n}_{\theta}|{\cal{G}}_t]\geq  E[X^{\mathbb{G}}_{\theta}|{\cal{G}}_t]- E[\sup_t(X^{\mathbb{G}}_t)^+I_{\{\theta\geq T^{\mathbb{G}}_n\}}|{\cal{G}}_t]\end{equation*}
implies that 
\begin{eqnarray*}
 \liminf_{n\longrightarrow+\infty}E[X^{\mathbb{G},n}_{\theta}|{\cal{G}}_t]\geq  E[X^{\mathbb{G}}_{\theta}|{\cal{G}}_t].\end{eqnarray*}
Hence, a combination of the latter inequality with (\ref{Fatou}) proves that $E[X^{\mathbb{G},n}_{\theta}|{\cal{G}}_t]$ converges to $E[X^{\mathbb{G}}_{\theta}|{\cal{G}}_t]$ almost surely. Similar arguments, allow us to prove the convergence almost surely of $E\left[ G_{\sigma}X_{\sigma}^{\mathbb{F},n}+\int_t^{\sigma} k^{(op,n)}_sdD^{o,\mathbb F}_s \big|{\cal F}_t\right]$  to $E\left[ G_{\sigma}X_{\sigma}^{\mathbb{F}}+\int_t^{\sigma} k^{(op)}_sdV^{\mathbb F}_s \big|{\cal F}_t\right]$. This proves assertion (a). The proof of assertion (b), under the assumption $E^{\widetilde{Q}}\left[\sup_{0\leq{t}\leq{T}} (X^{\mathbb G}_{t})^+\right]<\infty$, mimics exactly the proof of assertion (a), and will be omitted. This ends the proof of theorem.\end{proof}
\subsection{$\mathbb{G}$-optimal stopping times versus $\mathbb{F}$-optimal stopping times }
In this subsection, we investigate how the solutions to the optimal stopping problems under $\mathbb{G}$ and $\mathbb{F}$ are related to each other in many aspects. 
\begin{theorem}\label{MaximalMinimalTimesG2F}Suppose $G>0$, and let $X^{\mathbb G}$ be a RCLL $\mathbb G$-optional process of class-$(\mathbb{G},{\cal{D}})$ such that $(X^{\mathbb G})^{\tau}=X^{\mathbb G}$. Consider the unique pair $(X^{\mathbb F}, k^{(pr)})$ associated to $X^{\mathbb G}$ via Theorem \ref{PropositionG2F}, and put $\widetilde{X}^{\mathbb{F}}:=X^{\mathbb F}G+k^{(op)}\is{ D}^{o,\mathbb F}$, where $k^{(\rm{op})}$ is given in (\ref{k(F)}). Then the following assertions hold.\\
{\rm{(a)}} The optimal stopping problem for $(X^{\mathbb G},\mathbb{G})$ has a solution if and only if the optimal stopping problem for $(\widetilde{X}^{\mathbb F},\mathbb{F})$ has a solution. Furthermore, if one of these solutions exists, then the minimal optimal stopping times  $\theta^{\mathbb{G}}_*$  and $\theta^{\mathbb{F}}_*$, for $(X^{\mathbb G},\mathbb{G})$ and  $(\widetilde{X}^{\mathbb F},\mathbb{F})$ respectively, exist and $\theta^{\mathbb{G}}_*={{\min(\theta^{\mathbb{F}}_*,\tau)}}$.\\
{\rm{(b)}} The maximal optimal stopping time ${\widetilde\theta}^{\mathbb{G}}$ for  $(X^{\mathbb G},\mathbb{G})$ exists if and only if the maximal optimal stopping time ${\widetilde\theta}^{\mathbb{F}}$ for $(\widetilde{X}^{\mathbb F},\mathbb{F})$ exists also, and they satisfy  ${\widetilde\theta}^{\mathbb{G}}={{\min({\widetilde\theta}^{\mathbb{F}},\tau)}}$.
\end{theorem}
This theorem is established under strong integrability of $X^{\mathbb G}$. Thus, in our random horizon setting, the general framework of \cite{Quenez1} remains open. The proof of Theorem \ref{MaximalMinimalTimesG2F} relies essentially on Lemma \ref{L/EpsilonTilde}, and the following lemma which is interesting in itself.

\begin{lemma}\label{MinimalOST} Let $\mathbb{H}$ be a filtration,  $X$  be a RCLL and $\mathbb{H}$-adapted process of class-$(\mathbb{H},{\cal{D}})$, ${\cal{S}}^{\mathbb{H}}$  be its Snell envelop, and consider  
\begin{equation}
{\mathbb{T}}(X,\mathbb{H}):=\left\{\theta\in\mathcal{J}_0^{\infty}(\mathbb{H}):\ \Lbrack\theta\Rbrack\subset\{X={\cal{S}}^{\mathbb{H}}\},\ ({\cal{S}}^{\mathbb{H}})^{\theta}\in{\cal{M}}(\mathbb{H})\right\}.
\end{equation}
Then $\rm{ess}\inf{\mathbb{T}}(X,\mathbb{H})$ belongs to ${\mathbb{T}}(X,\mathbb{H})$ as soon as this set is not empty.
\end{lemma}
\begin{proof} Suppose that ${\mathbb{T}}(X,\mathbb{H})\not=\emptyset$, and remark that $\theta=\min(\theta_1,\theta_2)\in {\mathbb{T}}(X,\mathbb{H})$, for any $\theta_i\in {\mathbb{T}}(X,\mathbb{H})$, $i=1,2$, due to $\Lbrack\theta\Rbrack\subset \Lbrack\theta_2\Rbrack\cup \Lbrack\theta_1\Rbrack$. This implies that this set is downward directed, and hence there exists a non-increasing sequence $\theta_n\in  {\mathbb{T}}(X,\mathbb{H})$ such that  we have $\widetilde\theta:={\essinf}\ {\mathbb{T}}(X,\mathbb{H})={\inf_n}{\theta_n}$. It is obvious that $\Lbrack\widetilde\theta\Rbrack\subset\{X={\cal{S}}^{\mathbb{H}}\}$ due to the right-continuity of both $X$ and ${\cal{S}}^{\mathbb{H}}$. This proves the lemma.
\end{proof}
The remaining part of this subsection proves Theorem \ref{MaximalMinimalTimesG2F}.
\begin{proof}[Proof of Theorem \ref{MaximalMinimalTimesG2F}]  
The proof of this theorem is given in three parts.\\
{\bf Part 1.} This part proves the following fact:  
\begin{equation}\label{mainclaim}
\theta^{\mathbb{G}}\in {\mathbb{T}}(X^{\mathbb{G}},\mathbb{G})\ \mbox{iff there exists}\ \theta^{\mathbb{F}}\in {\mathbb{T}}(\widetilde{X}^{\mathbb{F}},\mathbb{F})\ \mbox{and}\ \theta^{\mathbb{G}}={{\min(\theta^{\mathbb{F}},\tau)}},\ P\mbox{-a.s.}.\end{equation}
Let $\theta^{\mathbb{G}}\in {\mathbb{T}}(X^{\mathbb{G}},\mathbb{G})$. Thus, there exists an $\mathbb{F}$-stopping time $\theta^{\mathbb{F}}$ such that 
\begin{equation}\label{T(G,XG)}
\Lbrack\theta^{\mathbb{F}}\wedge\tau\Rbrack\subset\{X^{\mathbb{G}}={\cal{S}}^{\mathbb{G}}\},\quad\mbox{and}\quad ({\cal{S}}^{\mathbb{G}})^{\theta^{\mathbb{F}}}\in {\cal{M}}(\mathbb{G}).\end{equation}
On the one hand, in virtue of Lemma \ref{L/EpsilonTilde}-(b) and (\ref{Snell4(G,P)}), remark that $X^{\mathbb{G}}_{\tau}=k^{(pr)}_{\tau}={\cal{S}}^{\mathbb{G}}_{\tau}$ $P$-a.s., and hence 
\begin{equation}\label{equality4T(G,XG)}
\{X^{\mathbb{G}}={\cal{S}}^{\mathbb{G}}\}=(\{X^{\mathbb{G}}={\cal{S}}^{\mathbb{G}}\}\cap\Lbrack0,\tau\Lbrack)\cup \Lbrack\tau\Rbrack.\end{equation}
On the other hand, by combining Lemma \ref{L/EpsilonTilde}-(b) (precisely the equality (\ref{V/G})) and (\ref{Snell4(G,P)}) again, we obtain
\begin{align*}
\{X^{\mathbb{G}}={\cal{S}}^{\mathbb{G}}\}\cap\Lbrack0,\tau\Lbrack&=\left\{X^{\mathbb{F}}=\left({\cal{S}}^{\mathbb{F}}-k^{(\rm{op})}\is D^{o,\mathbb{F}}\right)G^{-1}\right\}\cap\Lbrack0,\tau\Lbrack\\
&=\{{\cal{S}}^{\mathbb{F}} =GX^{\mathbb{F}}+k^{(\rm{op})}\is D^{o,\mathbb{F}}\}\cap\Lbrack0,\tau\Lbrack=\{{\cal{S}}^{\mathbb{F}} ={\widetilde{X}}^{\mathbb{F}}\}\cap\Lbrack0,\tau\Lbrack.
\end{align*}
Therefore, by combining this equality with(\ref{equality4T(G,XG)}) and
$$
\Lbrack{{\min(\theta^{\mathbb{F}},\tau)}}\Rbrack\cap \Lbrack0,\tau\Lbrack=\Lbrack \theta^{\mathbb{F}}\Rbrack\cap \Lbrack0,\tau\Lbrack,$$
we deduce that the first condition in (\ref{T(G,XG)}) is equivalent to 
\begin{equation*}
\Lbrack \theta^{\mathbb{F}}\Rbrack\cap \Lbrack0,\tau\Lbrack\subset \{{\cal{S}}^{\mathbb{F}} ={\widetilde{X}}^{\mathbb{F}}\}\cap\Lbrack0,\tau\Lbrack,\quad\mbox{or equivalently}\quad
I_{\Lbrack \theta^{\mathbb{F}}\Rbrack}I_{ \Lbrack0,\tau\Lbrack}\leq I_{\{{\cal{S}}^{\mathbb{F}} ={\widetilde{X}}^{\mathbb{F}}\}}I_{\Lbrack0,\tau\Lbrack}. 
\end{equation*}
Thus, by taking $\mathbb{F}$-optional projection,  and using $G={^{o,\mathbb{F}}}(I_{\Lbrack0,\tau\Lbrack})>0$, this latter condition is equivalent to  
\begin{equation}\label{Result1}
\Lbrack \theta^{\mathbb{F}}\Rbrack\subset \{{\cal{S}}^{\mathbb{F}} ={\widetilde{X}}^{\mathbb{F}}\}.\end{equation}
Thanks again to Theorem \ref{SnellEvelopG2F}-(a) and Lemma \ref{L/EpsilonTilde}-(a), we conclude that 
$$
({\cal{S}}^{\mathbb{G}})^{\theta}\in {\cal{M}}_{loc}(\mathbb{G})\ \mbox{iff}\ ({\cal{S}}^{\mathbb{F}}G^{-1}I_{\Lbrack0,\tau\Lbrack})^{\theta}\in {\cal{M}}_{loc}(\mathbb{G})\  \mbox{iff}\ {\cal{T}}(({\cal{S}}^{\mathbb{F}})^{\theta})\in {\cal{M}}_{loc}(\mathbb{G}),$$
for any $\mathbb{F}$-stopping time $\theta$. As $({\cal{S}}^{\mathbb{F}})^{\theta}$ is an $\mathbb{F}$-supermartingale, then there exist $M\in {\cal{M}}_{loc}(\mathbb{F})$ and a nondecreasing $\mathbb{F}$-predictable $A$ such that $({\cal{S}}^{\mathbb{F}})^{\theta}={\cal{S}}^{\mathbb{F}}_0+M-A$ and $M_0=A_0=0$. Thus,  ${\cal{T}}(({\cal{S}}^{\mathbb{F}})^{\theta})\in {\cal{M}}_{loc}(\mathbb{G})$ if and only if 
$${{G_{-}}\over{\widetilde{G}}}\is{A}^{\tau}-{^{p,\mathbb{F}}}(I_{\{\widetilde{G}=1\}})\is A^{\tau}={\cal{T}}(A)\in {\cal{M}}_{loc}(\mathbb{G}),$$
or equivalently its $\mathbb{G}$-compensator, which coincides with $ A^{\tau}$, is a null process. This implies that $G_{-}\is A\equiv 0$, and hence $A\equiv0$ due to $G>0$. Therefore, $({\cal{S}}^{\mathbb{F}})^{\theta}$ is an $\mathbb{F}$-local martingale. This proves the claim that 
\begin{equation}\label{Equivalence4GmgFmg}
({\cal{S}}^{\mathbb{G}})^{\theta}\in {\cal{M}}_{loc}(\mathbb{G})\ \mbox{iff}\ ({\cal{S}}^{\mathbb{F}})^{\theta}\in {\cal{M}}_{loc}(\mathbb{F}),\quad \mbox{for any $\mathbb{F}$-stopping}\ \theta.\end{equation}
Hence, the claim (\ref{mainclaim}) follows immediately from combining (\ref{T(G,XG)}), (\ref{Result1}) and (\ref{Equivalence4GmgFmg}). This ends part 1.\\
{\bf Part 2.} Herein, we prove assertion (a). Thanks to part 1, it is clear that $ {\mathbb{T}}(X^{\mathbb{G}},\mathbb{G})\not=\emptyset$ if and only if $ {\mathbb{T}}(\widetilde{X}^{\mathbb{F}},\mathbb{F})\not=\emptyset$. Thus, on the one hand, this proves the first statement in assertion (a). On the other hand, by combining this statement with Lemma \ref{MinimalOST}  and part 1 again, the second statement of assertion (a) follows immediately. \\
{\bf Part 3.} If the maximal optimal stopping time ${\widetilde\theta}^{\mathbb{F}}$ for $(\widetilde{X}^{\mathbb F},\mathbb{F})$ exists, then ${\widetilde\theta}^{\mathbb{F}}\in {\mathbb{T}}(\widetilde{X}^{\mathbb{F}},\mathbb{F})$ and for any $\theta\in {\mathbb{T}}(\widetilde{X}^{\mathbb{F}},\mathbb{F})$, we have $\theta\leq {\widetilde\theta}^{\mathbb{F}}=\esssup {\mathbb{T}}(\widetilde{X}^{\mathbb{F}},\mathbb{F})$ $P$-a.s.. Hence, for any $\theta^{\mathbb{G}}\in  {\mathbb{T}}(X^{\mathbb{G}},\mathbb{G})$, there exists $\theta\in {\mathbb{T}}(\widetilde{X}^{\mathbb{F}},\mathbb{F})$ satisfying
$$
\theta^{\mathbb{G}}={{\min(\theta,\tau)}}\leq {{\min({\widetilde\theta}^{\mathbb{F}},\tau)}}\in  {\mathbb{T}}(X^{\mathbb{G}},\mathbb{G}).$$
This proves that ${{\min({\widetilde\theta}^{\mathbb{F}},\tau)}}=\rm{ess}\sup{\mathbb{T}}(X^{\mathbb{G}},\mathbb{G})$, and hence the maximal optimal stopping time for $(X^{\mathbb{G}},\mathbb{G})$ exists and coincides with ${\widetilde\theta}^{\mathbb{F}}\wedge\tau$. To prove the converse, we suppose that the maximal optimal stopping time ${\widetilde\theta}^{\mathbb{G}}$ for  $(X^{\mathbb G},\mathbb{G})$ exists. Then in virtue of part 1, there exists $\theta^{\mathbb{F}}\in {\mathbb{T}}(\widetilde{X}^{\mathbb{F}},\mathbb{F})$ such that ${\widetilde\theta}^{\mathbb{G}}={{\min(\theta^{\mathbb{F}},\tau)}}$, $P$-a.s., and for any 
$\theta\in {\mathbb{T}}(\widetilde{X}^{\mathbb{F}},\mathbb{F})$, we have ${{\min(\theta,\tau)}}\in  {\mathbb{T}}(X^{\mathbb{G}},\mathbb{G})$ 
$$
\mbox{and}\quad {{\min(\theta,\tau)}}\leq {\widetilde\theta}^{\mathbb{G}}={{\min(\theta^{\mathbb{F}},\tau)}},\quad P\mbox{-a.s..}$$
This yields $\theta\leq \theta^{\mathbb{F}}\quad P\mbox{-a.s. on}\ (\theta<\tau),$ or equivalently 
$I_{\{\theta<\tau\}}\leq{I}_{\{\theta\leq \theta^{\mathbb{F}}\}}$ $P\mbox{-a.s..}$
By taking conditional expectation with respect to ${\cal{F}}_{\theta}$, on both sides of this inequality, we get $G_{\theta}\leq{I}_{\{\theta\leq \theta^{\mathbb{F}}\}}$ $P$-a.s.. and hence $\theta\leq \theta^{\mathbb{F}}$ $P$-a.s.. Therefore, we get $\esssup{\mathbb{T}}(\widetilde{X}^{\mathbb{F}},\mathbb{F})= \theta^{\mathbb{F}}\in {\mathbb{T}}(\widetilde{X}^{\mathbb{F}},\mathbb{F})$. Hence, the maximal optimal stopping time for $(\widetilde{X}^{\mathbb{F}},\mathbb{F})$,  denoted by ${\widetilde\theta}^{\mathbb{F}}$, exists and satisfies ${{\min(\tau,{\widetilde\theta}^{\mathbb{F}})}}= {\widetilde\theta}^{\mathbb{G}}$. This proves assertion (b), and completes the proof of the theorem.
\end{proof}
{{\section{Conclusion}
In this paper, we addressed various aspects of the optimal stopping problem in the setting where there are two flows of information. One ``public" flow $\mathbb{F}$ which is received by everyone in the system overtime, and a larger flow $\mathbb{G}$ containing additional information about the occurrence of a random time $\tau$. In this framework, our study starts by parametrizing in a unique manner (i.e. one-to-one parametrization) any $\mathbb{G}$-reward by processes and rewards that are $\mathbb{F}$-observable. Afterwards, we use this parametrization to single out the deep mathematical structures of the value process of the optimal stopping problem under $\mathbb{G}$, while highlighting the various terms induced by the randomness in $\tau$.  This resulting decomposition is highly motivated by the applications in risk management of the informational risks intrinsic to $\tau$. Furthermore, we established the one-to-one connection between $\mathbb{G}$-optimal stopping problem and its associated $\mathbb{F}$-optimal stopping problem, and we described the exact relationship between their maximal (respect. minimal) optimal times. Up to our knowledge, the obtained results are the first of their kind.

Besides this, our setting is the most general considered in the literature for the pair $(\mathbb{F},\tau)$. In fact, herein, we assume that the survival probability process $G$ is positive (i.e.$G>0$) only, while most of the literature (or all of it) assumes other assumptions such that the initial system (represented by $\mathbb{F}$) is Markovian and $\tau$ satisfies the immersion assumption, or the density assumption, or  the independence assumption between $\tau$ and $\mathbb{F}$, see \cite{Guo} and the references therein to cite a few. An other direction in the literature consists of addressing the optimal stopping problem with restricted information instead, and for this framework we refer the reader to \cite{Agram} and \cite{Oksendal0} and the references therein to cite a few.

Even though our setting is very general, it can be extended in two directions. The first extension consists of relaxing the assumption $G>0$ even though it is always assumed in the literature and is very acceptable in practice in contrast to the other assumptions. The second extension lies in allowing the reward process to have irregularities in its paths as in \cite{Quenez1}, and explore then how this irregularities interplay with the randomness in $\tau$. }}
\vspace{6pt} 

\appendix
\section{Proofs of Lemmas \ref{stoppingTimeLemma} and \ref{L/EpsilonTilde}}\label{Appendix4Proofs}
\begin{proof}[Proof of Lemma \ref{stoppingTimeLemma}]
Thanks to \cite[XX.75 b)]{DMM} (see also \cite[Proposition B.2-(b)]{Aksamit}), for a $\mathbb{G}$-stopping time $\sigma^{\mathbb{G}}$, there exists an  $\mathbb{F}$-stopping time $\sigma$ such that 
\begin{eqnarray*}
\sigma^{\mathbb{G}}=\sigma^{\mathbb{G}}\wedge\tau=\sigma\wedge\tau.
\end{eqnarray*}
Put 
$ \sigma^{\mathbb{F}}:=\min\left(\max(\sigma,\sigma_{1}),\sigma_{2}\right),$
  and on the one hand remark that $ \sigma^{\mathbb{F}}$ is an $\mathbb{F}$- stopping time satisfying the first condition in (\ref{sigmaF}). On the other hand, it is clear that 
  \begin{eqnarray*}
  \min(\tau, \max(\sigma, \sigma_1))=(\tau\wedge\sigma_1)I_{\{\sigma_1>\sigma\}}+(\tau\wedge\sigma)I_{\{\sigma_1\leq\sigma\}}=\max(\tau\wedge\sigma, \sigma_1\wedge\tau).\end{eqnarray*}
  Thus, by using this equality, we derive
  \begin{eqnarray*}
   \sigma^{\mathbb{F}}\wedge\tau&&= \tau\wedge\sigma_2\wedge\max(\sigma,\sigma_1)=(\tau\wedge\sigma_2)\wedge(\tau\wedge\max(\sigma,\sigma_1))\\
   &&=(\tau\wedge\sigma_2)\wedge\max(\tau\wedge\sigma, \sigma_1\wedge\tau)=\sigma\wedge\tau=\sigma^{\mathbb G}.
   \end{eqnarray*}
   This ends the proof of the lemma. \end{proof}
\begin{proof}[Proof of Lemma \ref{L/EpsilonTilde}] 
1) Herein, we prove assertion (a). Let $L$ be an $\mathbb F$-semimartingale. Then throughout this proof we put  $X:=L\widetilde{\cal E}^{-1}I_{\Lbrack0,\tau\Lbrack}$, and we derive
 \begin{align*}
 X&= {{L^{\tau}}\over{\widetilde{\cal E}^{\tau}}}- {{L}\over{\widetilde{\cal E}}}\is{ D}- {L_0}I_{\{\tau=0\}} =X_0+L\is {1\over{\widetilde{\cal E}^{\tau}}}+{1\over{\widetilde{\cal E}_{-}}}\is{ L}^{\tau}- {{L}\over{\widetilde{\cal E}}}\is{ D}\\
&= X_0+{{L}\over{G\widetilde{\cal E}_{-}}}I_{\Rbrack0,\tau\Rbrack}\is{D}^{o,\mathbb F} +{1\over{\widetilde{\cal E}_{-}}}\is{ L}^{\tau}- {{L}\over{\widetilde{\cal E}}}\is{ D}\\
&=X_0+  {{L}\over{\widetilde{G}\widetilde{\cal E}}}I_{\Rbrack0,\tau\Rbrack}\is{D}^{o,\mathbb F} +{1\over{\widetilde{\cal E}_{-}}}\is{ L}^{\tau}- {{L}\over{\widetilde{\cal E}}}\is{ D}=X_0- {{L}\over{\widetilde{\cal E}}}\is{N}^{\mathbb G}+{1\over{\widetilde{\cal E}_{-}}}\is{ L}^{\tau}.
\end{align*}
 The third equality follows from $d\widetilde{\cal E}^{-1}=\widetilde{\cal E}^{-1}_{-}G^{-1}d{D}^{o,\mathbb F}$, while the  fourth equality is due to $\widetilde{\cal E}=\widetilde{\cal E}_{-}G/\widetilde{G}$. A combination of the above latter equality with $X_0=L_0{I}_{\{\tau>0\}}$ proves (\ref{LoverEpsilonTilde}). \\
 2) This part proves assertion (b). To this end, in virtue of Lemma \ref{Decomposition4G}, we remark that $LG^{-1}I_{\Lbrack0,\tau\Lbrack}=G_0^{-1}{\widetilde{Z}}X$,  where ${\widetilde{Z}}$ is defined in (\ref{Ztilde}) and satisfies 
 \begin{equation*}
 {\widetilde{Z}}^{\tau}=1/{\cal{E}}(G_{-}^{-1}\is{m})^{\tau}={\cal{E}}(-G_{-}^{-1}\is{\cal{T}}(m)).\end{equation*}
Here, thanks to $G>0$ and Lemma \ref{Decomposition4G}, we used the fact that ${\cal{T}}(M)=M^{\tau}-\widetilde{G}^{-1}\is[M,m]^{\tau}$ for any $\mathbb{F}$-local martingale $M$. Thus, by applying It\^o to ${\widetilde{Z}}X$, we get 
 \begin{align*}
{\widetilde{Z}}X&={\widetilde{Z}}^{\tau}X=X_0+{\widetilde{Z}}_{-}\is{X}+X_{-}\is{\widetilde{Z}}+[X,{\widetilde{Z}}]\\
&=X_0+{\widetilde{Z}}_{-}\is{X}-X_{-}{\widetilde{Z}}_{-}G_{-}^{-1}\is{\cal{T}}(m)-{\widetilde{Z}}_{-}G_{-}^{-1}\is[X,{\cal{T}}(m)]\\
&=X_0+{\widetilde{Z}}_{-}\is{X}-X_{-}{\widetilde{Z}}_{-}G_{-}^{-1}\is{\cal{T}}(m)-{\widetilde{Z}}_{-}\widetilde{G}^{-1}\is[X,m]^{\tau}.
 \end{align*}
 Thus, by inserting (\ref{LoverEpsilonTilde}) into the above latter equality and using the facts $\widetilde{Z}_{-}/\widetilde{\cal E}_{-}=G_0/G_{-}$, $X_{-}=L_{-}\widetilde{\cal E}_{-}^{-1}I_{\Rbrack0,\tau\Rbrack}$, $\widetilde{Z}=\widetilde{Z}_{-}G_{-}/\widetilde{G}$ and $\widetilde{\cal E}=\widetilde{\cal E}_{-}G/\widetilde{G}$, we obtain
\begin{align*}
{\widetilde{Z}}X&=X_0 - {{G_0{\widetilde{G}}L}\over{GG_{-}}}\is{N}^{\mathbb G}+{{G_0}\over{G_{-}}}\is{ L}^{\tau}-{{G_0L_{-}}\over{G_{-}^2}}\is{\cal{T}}(m)-G_0{\widetilde{G}}^{-1}\widetilde{G}^{-1}\is[L,m]^{\tau}+{{G_0{L}\Delta{m}}\over{GG_{-}}} \is{N}^{\mathbb G}\\
&=X_0 - {{G_0L}\over{G}}\is{N}^{\mathbb G}+{{G_0}\over{G_{-}}}\is{\cal{T}}(L)-{{G_0L_{-}}\over{G_{-}2}}\is{\cal{T}}(m).  \end{align*}
  Therefore, by combining this latter equality with $LG^{-1}I_{\Lbrack0,\tau\Lbrack}=G_0^{-1}\widetilde{Z}X$, the equality (\ref{LoverG}) follows immediately, and the proof of the lemma is complete.
\end{proof}

\end{document}